\documentclass[12pt]{article}
\usepackage[T1]{fontenc}
\usepackage[utf8]{inputenc}
\usepackage[english]{babel}
\usepackage{authblk}

\usepackage{hyperref}
\usepackage{lmodern}
\usepackage{amsthm}
\usepackage{pdfpages}
\usepackage{listings}
\usepackage{multirow}
\usepackage{color, colortbl}
\usepackage{colortbl}
\usepackage{amsmath}
\usepackage{setspace}
\usepackage[noend]{algpseudocode}
\usepackage{enumerate}
\usepackage{mathtools,amssymb,amsthm}
\usepackage{color}
\usepackage{enumitem}
 \usepackage{mathtools}
 \usepackage{epigraph}
\usepackage{fancyhdr}
\usepackage{amsfonts}
\usepackage{mathrsfs}
\usepackage{tikz-cd}
\usepackage{fancybox}
\usepackage[all]{xy}
\usepackage{pifont}
\newtheorem{thm}{Theorem}[section]
\newtheorem{cor}{Corollary}[section]

\newtheorem{ex}{Example}[section]
\newtheorem{remark}{Remark}[section]

\newtheorem{lemma}{Lemma}[section]
\newtheorem{defo}{Definition}[section]

\newtheorem{prop}{Proposition}[section]

\usepackage{blindtext}
\title{\textbf{ON GRADED U-NIL CLEAN RINGS}}
\author{ I. Namrok\footnote{namrokismail@gmail.com} }
\affil{}
\date{}

\begin{document}
\maketitle
\begin{abstract} In this paper we introduce and study the class of graded U-nil clean rings, as a generalization of graded nil-good class defined in \cite{5}. We also investigate the transfer of the graded U-nil cleaness to  matrix rings, and to graded group rings as well. The question of when  trivial ring extensions are graded U-nil clean is also investigated. \end{abstract}
\vskip0.1in
\textbf{keywords}: Graded rings and modules, U-nil clean ring, group ring, matrix ring, idealization of a module.
\section{Introduction}

The year 1977 was the beginning of a new era of the introduction of new classes of rings, which generalize the concept of clean rings introduced by Nicholson \cite{11} in the same year. Indeed, Nicholson defined a clean ring as a ring in which every element is a sum of an idempotent and a unit. Ever since, and particularly over the last decade, many authors  have studied numerous  classes of rings as a generalization or a variation of clean rings; for instance, in 2016 Danchev  introduced in \cite{12}, an interesting class, called nil-good rings, in which every element is either nilpotent or a sum of a nilpotent and a unit. Danchev has also given   significant results related to extensions of such rings, such as group rings and matrix rings as well. In his turn, Khashan took the initiative in  \cite{3}, to enrich the literature related to  classes of rings, by introducing the class of U-nil clean rings, whose definition is given in the next section. The aim of this paper is studying the U-nil clean property from  graded ring theory point of view. Indeed, the idea of extending the properties of  rings to graded rings has been interesting many authors recently; for more mdetails, the reader is referred to \cite{5,19,16,17,15}. In particular, in \cite{5} the authors examined the nil-good property in graded rings; indeed in this paper, we introduce graded U-nil clean rings as a graded version of the concept given by Khashan in \cite{3} and as a generalization of the  graded nil-good notion defined in \cite{5}. We begin by giving basic properties and examples of graded U-nil clean rings. In addition, we establish a sufficient conditions for a graded ring to be graded U-nil clean, if the component which corresponds to the identity element of the grading  group is a U-nil clean ring. Then, we emphasis on the extensions of such rings, in particular the matrix rings, since it is the main goal of this paper. We also study the question of when group rings and trivial ring extensions are graded U-nil clean.

\section{Preliminaries}
Throughout this paper, all rings are assumed to be unital and associative. If $R$ is a ring then $J(R)$ denotes the Jacobson radical of $R$, $U(R)$ is the multiplicative group of units of $R$, $N(R)$ denotes the set of nilpotent elements of $R$, $Z(R)$ is the center of the ring $R$ and $Idem(R)$ stands for the set of idempotents of $R$.

Let $R$ be a ring, $G$ a group with identity element $e$, and let $\{R_g\}_{g\in G}$ be a family of additive subgroups of $R$. the ring $R$ is said to be $G$-graded (or simply graded) if $R= \bigoplus_{g\in G}R_g$ and $R_gR_h\subseteq R_{gh}$ for every $g,h\in G$. Let $h(R)= \bigcup_{g\in G}R_g $; elements of $h(R)$ are said homogeneous, and for every $g\in G$, the subgroup $R_g$ is called the $g$-th component of $R$. If $a\in R_g$, then we say that $a$ is of degree $g$. In addition, the degree of every homogeneous idempotent element is $e$. Furthermore, if $t\in Idem(R_e)$, then $tR$ is a graded ring with unity $t$ and the $g$-component $(tR)_g=tR_g$ for each $g\in G$. The set $sup(R)=\{g\in G, R_g \neq 0\}$ is called the \textit{support} of the graded ring $R$.

If $A=\oplus_{g\in G}A_g$ and $B=\oplus_{g\in G}B_g$ are  two $G$-graded rings;  a ring homomorphism $\epsilon:A \longrightarrow B$ is said to be \textit{graded-homomorphism} if $\epsilon(h(A))  \subseteq h(B)$. The ring homomorphism $\epsilon$ is called \textit{degree-preserving homomorphism} if $\epsilon(A_g)\subseteq B_g$ for every $g\in G$. 

Given $G$-graded rings $R_1,R_2,\dots,R_k$. The direct product $R=\prod_{i=1}^kR_i$ is a $G$-graded ring, where its $g$-th component is $R_g=\prod_{i=1}^kR_i^g$, with $R_i^g$ is the $g$-th component of the ring $R_i$. 

A proper right (left or two-sided) ideal $I$ of a $G$-graded ring $R= \oplus_{g\in G}R_g$ is called \textit{homogeneous} (or \textit{graded}), if $I= \oplus_{g\in G}I\cap R_g$. If $I$ is a two-sided homogeneous ideal, then $R/I$ is a $G$-graded ring where its $g$-th component is $(R/I)_g=R_g/I\cap R_g$. In addition, a homogeneous ideal is called \textit{graded-nil}, if every homogeneous element of this ideal is nilpotent. A homogeneous ideal $I$ of a graded ring $R$ is called \textit{graded-maximal}, if there is no other proper homogeneous ideal of $R$ which contains $I$; equivalently, according to page 46 in \cite{4}, $I$ is graded-maximal if $I+Rx=R$ for any $x\in h(R)\backslash I$. A graded ring is said to be \textit{graded-local}, if it has a unique graded-maximal right ideal. Moreover, the intersection of all graded-maximal right ideals of a graded ring $R$ defines the \textit{graded Jacobson radical} of $R$, denoted $J^g(R)$. Further, according to Proposition 2.9.1 in \cite{4} we have that, $J^g(R)$ is a homogeneous two-sided ideal.

If $R= \oplus_{g\in G}R_g$ is a $G$-graded ring, then a  $G$-graded $R$-module is an $R$-module $M$ such that $M= \oplus_{g\in G}M_g$, where $\{M_g\}_{g\in G}$ are additive subgroups of $M$, and $R_hM_g\subseteq M_{hg}$ for all $g,h\in G$.

Let $R= \oplus_{g\in G}R_g$ be a $G$-graded ring. According to \cite{8}, we have that the group ring $R[G]$ is $G$-graded where its $g$-th component is $(R[G])_g= \oplus_{h\in G}R_{gh^{-1}}h$ and with the multiplication defined by $(r_gg')(r_hh')=r_gr_h(h^{-1}g'hh')$, where $g,g',h$ and $h' \in G$, $r_g \in R_g$ and $r_h\in R_h$. Now, if $H$ is a normal subgroup of $G$, then according to \cite{4} $R[H]$ can be viewed as a $G$-graded ring, where its $g$-th component is $(R[H])_g=\oplus_{h\in H}R_{gh^{-1}}h$. In addition, the ring $R$ can be seen as a $G/H$-graded ring (see page 178 in \cite{4}) as follow $R=\bigoplus_{C\in G/H}R_C$ where $R_C=\oplus_{x\in C}R_x$. Clearly $R[H]$ can also be observed as a $G/H$-graded ring.

Let $A$ be a ring and $E$ an $(A,A)$-bimodule. We denote by $R:=A\propto E$ the set of pairs $(a,e)$ with pairwise addition and multiplication given by $(a,e)(b,f)=(ab,af+eb)$; the ring $R$ is called the \textit{idealization} of $E$ (over $A$). Considerable background concerning the idealization of modules over commutative rings can be found in \cite{7,20,21}. If $A$ is a $G$-graded ring  and $E$ a $G$-graded $(A,A)$-bimodule, then according to \cite{7} $R=A\propto E$ is $G$-graded with the $g$-th component $R_g=A_g\oplus E_g$.

If $R=\oplus_{g\in G}R_g$ is a $G$-graded ring, $n$ an arbitrary natural number and $\sigma=(g_1,\dots,g_n)\in G^n$. For every $\lambda \in G$, we denote by $M_n(R)_{\lambda}(\sigma)$ the additive subgroup of $M_n(R)$ formed by matrices $(a_{ij})_{n\times n}$, such that $a_{ij}\in R_{g_i\lambda g_j^{-1}}$ for $i,j \in \{1,\dots,n\}$. We have $M_n(R)= \bigoplus_{\lambda \in G}M_n(R)_{\lambda}(\sigma)$ which makes the matrix ring $M_n(R)$ a $G$-graded ring with respect to usual matrix addition and multiplication. Throughout this paper, if the matrix ring $M_n(R)$ is observed as a graded ring in the previous way then we denote it by $M_n(R)(\sigma)$.

An element $r$ of a ring $R$ is called \textit{unit regular} if $r=eu$ for some $e \in Idem(R)$ and $u\in U(R)$. Khashan \cite{3} defined a ring to be U-nil clean, if every element is a sum of a nilpotent and a unit regular. Obviously, idempotents and units are unit regular elements. Let $R=\oplus_{g\in G}R_g$  be a $G$-graded ring; we define a homogeneous element $r\in R_g$ (where $g\in G$) to be \textit{graded unit regular} if $r=eu$ for some $e\in Idem(R_e)$ and $u\in U(R)\cap R_g$. The symbol $gur(R)$ will stand for the set of graded unit regular elements of $R$; while $ur(R)$ is the set of unit regular elements of $R$.

\section{Graded U-nil clean rings}

Let $G$ be a group with identity $e$.

\begin{defo}\label{d1}
A homogeneous element $x$ of a $G$-graded ring $R$ is called graded U-nil clean if $x=r+n$ where $r\in gur(R)$ and $n\in N(R)\cap h(R)$. The ring $R$ is called graded U-nil clean if every  element of $h(R)$ is graded U-nil clean.
\end{defo}

\begin{remark}\label{r1}
If $x=r+n$ is a graded U-nil clean decomposition of a homogeneous element $x$ of the ring $R$. Then, we can easily show that $x$, $r$ and $n$ are all of the same degree. In fact, let $g$ (resp. $h$) be the degree of $x$ (resp. $r$). If $g\neq h$, then we will have either $x=n$ or $x=0$. Both cases yield to the contradiction $g=h$. Therefore, $g=h$ and necessarily the degree of $n$ is also $g$.
\end{remark}

\begin{ex}\label{e1}
Let $A$ be any U-nil clean ring (for instance $\mathbb{Z}_3$ since it is a field), and let $G=\{e,g,h\}$ be a cyclic group of order three. The ring $R=M_2(A)$ of $2 \times 2$ matrices with coefficients in $A$ is clearly $G$-graded, with the grading: $R_e= \begin{pmatrix} A & 0 \\ 0 & A \end{pmatrix}$, $R_g=\begin{pmatrix} 0 & A \\ 0 & 0 \end{pmatrix}$ and $R_h=\begin{pmatrix} 0 & 0 \\ A & 0 \end{pmatrix}$. It is obvious that $R_e$ is U-nil clean. Since every element of $R_g$ and $R_h$ is nilpotent, then $R$ is a graded U-nil clean ring.
\end{ex}

The following example shows that the class of graded nil-good rings is a proper subclass of graded U-nil clean rings class. According to \cite{5}, a graded ring is said to be \textit{graded nil-good} if every homogeneous element is either nilpotent, or a sum of a homogeneous unit and a homogeneous nilpotent.

\begin{ex}\label{e3}
Consider $R=M_2(\mathbb{Z}_2)$ and $G$ a cyclic group of order two. Then, $ R=\begin{pmatrix} 
\mathbb{Z}_2 & 0 \\
0& \mathbb{Z}_2
\end{pmatrix} \bigoplus \begin{pmatrix} 
0 & \mathbb{Z}_2 \\
\mathbb{Z}_2 & 0
\end{pmatrix} $ is a $G$-graded ring which is not graded nil-good by \cite[Example 2.6]{5}. On the other hand, the homogeneous elements $\begin{pmatrix}
1 & 0 \\ 0 & 0
\end{pmatrix}$ and $\begin{pmatrix}
0 & 0 \\ 0 & 1
\end{pmatrix}$ are idempotents. Moreover, $\begin{pmatrix}
0 & 1 \\ 0 & 0
\end{pmatrix}$ and $\begin{pmatrix}
0 & 0 \\ 1 & 0
\end{pmatrix}$ are both nilpotents. Further, $A=\begin{pmatrix}
0 & 1 \\ 1 & 0
\end{pmatrix}$ is a unit since $A^2=\begin{pmatrix}
1 & 0 \\ 0 & 1
\end{pmatrix}$. Hence, every homogeneous element of $R$ is graded U-nil clean; thus $R$ is graded U-nil clean.
\end{ex}

\begin{prop}\label{p1}
Let $R=\oplus_{g\in G}R_g$ be a $G$-graded U-nil clean ring. Then,
\begin{enumerate}
    \item $R_e$ is a U-nil clean ring.
    \item Every graded-homomorphic image of $R$ is $G$-graded U-nil clean.
\end{enumerate}
\end{prop}

\begin{proof}
\begin{enumerate}
    \item Let $a \in R_e$. Since $R$ is graded U-nil clean, then $a=r+n$ where $r\in gur(R)$ and $n \in N(R)\cap h(R)$. According to Remark \ref{r1} we have $r$ and $n$ are both elements of $R_e$. Therefore, the ring $R_e$ is U-nil clean.
    \item Consider $S$ a G-graded ring and $\varphi :R \longrightarrow S$  a graded-homomorphism. Suppose that $R$ is graded U-nil clean and let $a\in h(R)$; our goal is to show that $\varphi(a)$ is  graded U-nil clean of the ring $\varphi(R)$. By assumption $a$ can be written $a=fu+n$ where $ f \in Idem(R_e)$, $u \in U(R)\cap h(R)$ and $n \in N(R)\cap h(R)$. Hence,  $\varphi(a)=\varphi(f)\varphi(u)+\varphi(n)$. Since $\varphi$ is a graded-homomorphism, then $\varphi(a)$ has been written as a sum of a graded unit regular element and a homogeneous nilpotent element of the ring $\varphi(R)$. Therefore, $\varphi(R)$ is graded U-nil clean.
\end{enumerate}
\end{proof}

The first point of Proposition \ref{p1}, makes us wondering if the U-nil clean property of $R_e$ implies that $R=\oplus_{g \in G}R_g$ is graded U-nil clean. Well, the next example proves that is false in general.

\begin{ex}\label{e4}
Let $R:=\mathbb{Z}_2[X]$. We have that $R=\oplus_{n\in \mathbb{Z}}R_n$ where $R_n=\mathbb{Z}_2X^n$ for $n \geq 0$ and $R_n=0$ for $n < 0$. Hence, $R$ is a $\mathbb{Z}$-graded ring. Obviously, $R_0=\mathbb{Z}_2$ is a U-nil clean ring since it is a field. On the other hand, it is clearly that the homogeneous element $X \in R_1$ is not graded U-nil clean. Hence, $R$ is not graded U-nil clean.
\end{ex}

\begin{prop}\label{p4}
Let $R_1,R_2,\dots,R_k$ be $G$-graded rings. If $R_i$ is graded U-nil clean for every $i\in \{1,2,\dots,k\}$, then $R=\prod_{i=1}^kR_i$ is a $G$-graded U-nil clean ring.
\end{prop}
\begin{proof}
Let $(r_1,r_2,\dots,r_k)$ be a homogeneous element of $R$ of degree $g$ where $g\in G$. For each $i$ we choose $a_i \in gur(R_i)$ and $n_i\in N(R_i)\cap h(R_i)$ such that $r_i=a_i+n_i$. It is clear that $a=(a_1,a_2,\dots,a_k)\in ur(R)\cap h(R)$ and $n=(n_1,n_2,\dots,n_k)\in N(R)\cap h(R)$; it follows that $r=a+n$ and therefore $R$ is $G$-graded U-nil clean.
\end{proof}

It has been proved in \cite[Proposition 2.1]{3} by Khashan that any U-nil clean ring is a $G$-clean ring ($G$ refers to the word "generalization" and not to the group mentioned in the head of this section). Let us recall first the definition of a $G$-clean ring introduced and studied by Hongbo and Weting in \cite{23}. A ring is defined to be $G$-clean if every element is a sum of a unit and a unit regular. The class of $G$-clean rings is a generalization of the class of clean rings introduced by Nicholson. In graded ring theory, if $H$ is a group; we define naturally in this paper, a $H$-graded ring to be graded $G$-clean if every homogeneous element is a sum of a homogeneous unit and a graded unit regular element.  The next result gives a graded version of  Proposition 2.1 in \cite{3}.

\begin{prop}\label{p2}

Let $H$ be a group with identity $e'$ and $R=\oplus_{g \in H} R_g$ a $H$-graded  commutative ring. If $R$ is graded U-nil clean, then it is a graded $G$-clean ring.
\end{prop}

\begin{proof}
Suppose that $R$ is graded U-nil clean, then by Proposition \ref{p1} $R_{e'}$ is a U-nil clean ring. Thus, according to \cite[Proposition 2.1]{3} the ring $R_{e'}$ is a $G$-clean ring. Now, let $a\in R_g$ (where $g\in H)$ be a homogeneous element of $R$. Since $R$ is graded U-nil clean, then there exist $r\in gur(R)\cap R_g$ and $n\in N(R)\cap R_g$ such that $a=r+n$. We know that there exist $x \in U(R)\cap R_g$  and  $f \in Idem(R_{e'})$ such that $r=fx$. Now, since $R_{e'}$ is $G$-clean, then $f=s+u$, where $s\in ur(R_{e'})$   and  $u\in U(R_{e'})$. Thus, $a$ is written  $a=sx+(ux+n)$. Since $R$ is commutative, then $ux+n \in U(R)\cap R_g$. Moreover, $s=fv$ for some $f\in Idem(R_{e'})$ and $v \in U(R_{e'})$. Hence, $sx=f(vx)$ is then graded unit regular of degree $g$. Finally, $a$ has been written as a sum of graded unit regular and a homogeneous unit, which means that $R$ is $H$-graded $G$-clean.
\end{proof}

\begin{remark}
It is well known that, if $I$ is a nil ideal of a ring $R$ then, the ideal $I$ is contained in the classical Jacobson radical of $R$ (i.e. $I\subseteq J(R))$. Next, we will give a graded approach of this result.
\end{remark}

\begin{lemma}\label{l2}
Let $R=\oplus_{g\in G} R_g$ be a $G$-graded ring and $I$ a  homogeneous ideal of $R$. If $I$ is graded-nil then $I\subseteq J^g(R)$.
\end{lemma}

\begin{proof}
Suppose that $I$ is a graded-nil ideal. Let $a$ be a homogeneous element of $I$; since $I$ is a proper ideal then $Ra\neq R$. Assume that $a \notin J^g(R)$, hence there exists a graded-maximal ideal $M$ of $R$ such that $a \notin M$. Therefore, as mentioned in the preliminaries of this paper, we have that $M+Ra=R$. Thus, $m+ra=1$ for some $m\in M$ and $r\in R$. Hence, $\overline{r}\overline{a}=\overline{1}$ in $R/M$ which is a contradiction since $\overline{a}$ is nilpotent in $R/M$. Therefore $I \subseteq J^g(R)$.
\end{proof}

\begin{thm}\label{t1}
Let $R=\oplus_{g \in G} R_g$ be a $G$-graded ring and $I$ a graded-nil ideal of $R$. Then $R$ is a graded U-nil clean ring if and only if so is $R/I$.
\end{thm}

\begin{proof}
''$\Longrightarrow$'' Suppose  $R$ is graded U-nil clean. The canonical map $\pi: R \longrightarrow R/I$ is clearly a graded-homomorphism. Now, according to proposition \ref{p1} $R/I$ is graded U-nil clean as a the graded-homomorphic image of $R$ by $\pi$.

''$\Longleftarrow$'' Let us suppose that $\overline{R}=R/I$ is graded U-nil clean and let $a\in R_g$ where $ g\in G$. There exist then $\overline{f}\in Idem(\overline{R}_e)$, $\overline{u} \in U(\overline{R})\cap \overline{R}_g$ and $\overline{n} \in N(\overline{R})\cap \overline{R}_g$ such that $\overline{a}=\overline{f}\overline{u}+\overline{n}$. On the other hand, $I_e=I\cap R_e$ is clearly a nil ideal of the ring $R_e$, hence by applying \cite[Proposition 27.1]{2}  to $R_e$, we have that $\overline{f}$ can be lifted to an idempotent element of $R_e$, thus we can suppose that $ f\in Idem(R_e)$. Now, since $I$ is graded-nil, then according to Lemma \ref{l2} $I \subseteq J^g(R)$; therefore, $\overline{u} \in U(R/J^g(R))$. According to \cite[Proposition 2.9.1]{4}, we have that $u \in U(R)$. We know that $\overline{a-fu}$ is a homogeneous nilpotent element of $R/I$. Thus, since $I$ is graded-nil, it follows that $a-fu$ is a homogeneous nilpotent element of $R$. Finally, $a$ is graded U-nil clean, and the proof is achieved. 
\end{proof}

Concerning the condition ''$I$ is graded-nil'' in Theorem \ref{t1}, it can't be neglected. Indeed, let $S=\mathbb{Z}_3[X]$; we have that $S=\oplus_{n\in \mathbb{Z}}S_n$  where $S_n=\mathbb{Z}_3X^n$ for $n \geq 0$ and $S_n=0$ for $n < 0$, which means that $S$ is $\mathbb{Z}$-graded. It is clear that $S$ is not graded U-nil clean (the homogeneous element $X$ is not graded U-nil clean). Now, the homogeneous ideal $I=(X)$ of $S$ is clearly not graded-nil. Yet, the ring $S/I$ is graded U-nil clean as an isomorphic image of the field $\mathbb{Z}_3$.

\begin{prop}\label{p3}

Let $R=\oplus_{g \in G} R_g$ be a $G$-graded U-nil clean ring. If $Idem(R_e) \subseteq Z(R)$, then $J^g(R)$ is graded-nil.
\end{prop}

\begin{proof}
Let $j$ be a homogeneous element of $J^g(R)$ of degree $g\in G$. Our goal is to prove that $j\in N(R)$. Indeed, according to Proposition \ref{p1} $R_e$ is U-nil clean; further, since $Idem(R_e) \subseteq Z(R)$ then $R_e$ is an abelian ring (i.e. $Idem(R_e) \subseteq Z(R_e)$). Hence, by  \cite[Porposition 2.3]{3} the Jacobson radical of $R_e$, $J(R_e)$ is nil. If $g=e$, then $j \in J^g(R)\cap R_e$. In addition, according to \cite[Corollary 2.9.3]{4} $J^g(R)\cap R_e=J(R_e)$; therefore $j$ is nilpotent. Now, suppose that $g \neq e$. Since $R$ is graded U-nil clean, then there exist $ f\in Idem(R_e)$, $u \in U(R)\cap R_g$ and $n \in N(R)\cap R_g$ such that $j=fu+n$. Let $k \in \mathbb{N}$ such that $n^k=0$. Using the binomial expansion we obtain $(fu-j)^k=(fu)^k+jx+yj=0$ for some $x,y \in R$. The condition $Idem(R_e) \subseteq Z(R)$ allows us to write $(fu)^k=f^ku^k$; moreover $f^k=f$ since $f$ is an idempotent element. Hence, since $J^g(R)$ is a two-sided ideal, we have $fu^k=-jx-yj \in J^g(R)$. Thus, $f=fu^ku^{-k} \in J^g(R) \cap R_e=J(R_e)$. But $J(R_e)$ is a nil ideal; therefore the idempotent element $f$ is  nlipotent. It follows $f=0$  and then $j=n \in N(R)$. Therefore, $J^g(R)$ is graded-nil.
\end{proof}

By combining Theorem \ref{t1} and Proposition \ref{p3} we obtain the following corollary.

\begin{cor}\label{c1}
Let $R=\oplus_{g \in G} R_g$ be $G$-graded ring such that $Idem(R_e) \subseteq Z(R)$. Then, $R$ is graded U-nil clean if and only if so is $R/J^g(R)$.
\end{cor}

Let's go back to Example \ref{e4}, which proves that the following implication:
$$ R_e ~\text{is U-nil clean}\Longrightarrow R= \oplus_{g\in G}R_g~\text{is graded U-nil clean}.$$ is false in general. Indeed, the above implication holds true under some sufficient conditions given below. Let us recall first the definition of a $PI$-ring. A ring $R$ is called $PI$-ring if there exist a natural integer $n$, and an element $P$ of $\mathbb{Z}[X_1,X_2,\dots,X_n]$ such that for every $(r_1,r_2,\dots,r_n)\in R^n$, we have $P(r_1,r_2,\dots,r_n)=0$. 

\begin{thm}\label{t2}
Let $R=\oplus_{g\in G} R_g$ be a graded-local $PI$-ring of finite support. If $R_e$ is U-nil clean, then $R$ is graded U-nil clean.
\end{thm}
\begin{proof}
Suppose that $R_e$ is U-nil clean. According to Proposition 2.3 in  \cite{3}, the Jacobson radical $J(R_e)$ of $R_e$ is nil. Since $R$ is a $PI$-ring then according to \cite[Theorem 3]{6} the Jacobson radical $J(R)$ of $R$ is also nil. Now, let $a$ be a homogeneous element of $R$. We have indeed two cases:\\
- Case 1: $a\notin J^g(R)$. Hence, the homogeneous ideals $aR$ and $Ra$ cannot be contained in $J^g(R)$. Since $R$ is graded-local, then $aR=Ra=R$. Thus, $a$ is a unit and in particular $a$ is graded U-nil clean.\\
- Case 2: $a\in J^g(R)$. Since the support of $R$ is finite, according to Corollary 2.9.4 in \cite{4} we have that $J^g(R) \subseteq J(R)$, and hence $J^g(R)$ is a nil ideal; therefore, $a$ is nilpotent.\\
Finally, every homogeneous element is graded U-nil clean; thus $R$ is graded U-nil clean.
\end{proof}

\section{Extensions of graded U-nil clean rings}

In this section, we'll discuss the ''graded U-nil clean'' property of graded group rings as well as of trivial ring extensions. The question of when the matrix ring is graded  U-nil clean is also investigated.

\subsection{Trivial ring extensions}

In this subsection, we investigate the behaviour of the graded U-nil clean property in trivial ring extensions of graded U-nil clean rings.

\begin{thm}\label{t3}
Let $A$ be a $G$-graded ring, $E$ a $G$-graded $(A,A)$-bimodule and $R=A\propto E$ the idealization of $E$. Then, $A$ is graded U-nil clean if and only if $R$ is graded U-nil clean.
\end{thm}

\begin{proof}
It is easy to check that $0\propto E$ is a homogeneous ideal of $R$; moreover, it is a graded-nil ideal since $(0\propto E)^2=0$. On the other hand, the epimorphism $\epsilon:R \longrightarrow A$ given by $\epsilon((a,e))=a$, is a graded-homomorphism. It is clear that the kernel of $\epsilon$ is exactly $0 \propto E$; hence $R/(0\propto E)$ and $A$ are graded-isomorphic. Now, the Theorem \ref{t1} completes the proof.
\end{proof}

The previous theorem allows us give more examples in addition the one given in Example \ref{e3}, of graded U-nil clean rings which are not graded nil-good. Let's consider the graded ring $R$ taken in Example \ref{e3}; and let $E$ be any $G$-graded $(R,R)$-bimodule (where $G$ is the cyclic group of order two). Hence, the graded ring $R\propto E$ is graded U-nil clean (by Theorem \ref{t3}), which is not graded nil-good according to theorem 3.1 in \cite{5}.

\subsection{Group rings}

We continue  our investigation of extensions of graded U-nil clean rings; this subsection will be dedicated to graded group rings.

\begin{lemma}\label{l1}
Let $A=\oplus_{g\in G} A_g$ and $B=\oplus_{g\in G} B_g$ be two G-graded rings. If $\varphi:A \longrightarrow B $ is a degree-preserving homomorphism, then its kernel $\ker(\varphi)$ is a homogeneous ideal of the ring $A$.
\end{lemma}

\begin{proof}
It is well known that $\ker(\varphi)$ is an ideal of $A$. Hence, it remains to show that $\ker(\varphi)$ is a homogeneous ideal. Let $\sum_{g\in G} a_g \in \ker(\varphi)$; therefore $0=\varphi(\sum_{g\in G} a_g) = \sum_{g\in G} \varphi(a_g)$. Since  $ \forall g \in G$: $ \varphi(a_g) \in B_h$ for some $h\in G$, then $\varphi(a_g)=0$ for every $g\in G$; it follows that $a_g \in \ker(\varphi)$ for each $g\in G$. Hence, $\ker(\varphi)$ is a homogeneous ideal of $A$.
\end{proof}

\begin{thm}\label{t4}
Let $R=\oplus_{g\in G} R_g$ be a $G$-graded ring. Then, if $R[G]$ is graded U-nil clean then $R$ is a U-nil clean ring.
\end{thm}

\begin{proof}
Suppose that $R[G]$ is graded U-nil clean. According to Proposition \ref{p1}, we have that $(R[G])_e$ is a U-nil clean ring. Furthermore, Proposition 2.1 in \cite{8} proves that the two rings $(R[G])_e$ and $R$ are isomorphic. Hence, by Proposition 2.2 in \cite{3} $R$ is a U-nil clean ring.
\end{proof}

\begin{thm}
Let $R$ be a $G$-graded ring, and $H$ a normal subgroup of $G$. Suppose that $G$ is locally finite $p$-group where $p$ is a nilpotent element of $R$. If $R$ is a $G/H$-graded U-nil clean ring then $R[H]$ is also $G/H$-graded U-nil clean.
\end{thm}

\begin{proof}
Suppose that $R$ is $G/H$-graded U-nil clean. As has been shown in the proof of Theorem 2.3 in \cite{9} , we can admit that $H$ is a finite $p$-group, hence $H=\{h_1,\dots,h_n\}$ where $h_i\in H$ for each $1\leq i \leq n$. Now, Let's consider the epimorphism $\epsilon: R[H] \longrightarrow R$ given by $\epsilon(\sum_{i=1}^n r_ih_i)=\sum_{i=1}^n r_i$, where $r_i \in R$ for every $ 1 \leq i \leq n$. It has been proved in page 180 of \cite{4} that the epimorphism $\epsilon$ is a degree-preserving homomorphism of $G/H$-graded rings. Therefore, by Lemma \ref{l1} the kernel of $\epsilon$ is a graded ideal of the ring $R[H]$. It follows that $R[H]/\ker(\epsilon)$ and $R$ are $G/H$-graded-isomorphic rings; therefore, $R[H]/\ker(\epsilon)$ is $G/H$-graded U-nil clean. Moreover, since $p \in N(R)$ then according to \cite[Theorem 9]{10} we have $\ker(\epsilon)$ is a nilpotent ideal of $R[H]$; hence $\ker(\epsilon)$ is a graded-nil ideal. Thus, by Theorem \ref{t1} $R[H]$ is $G/H$-graded U-nil clean.
\end{proof}

\subsection{Matrix rings}

As mentioned in the introduction; the main result of this paper, which is the U-nil clean property of graded matrix ring will be discussed in this subsection

\begin{thm}
Assume that $G$ is an Abelian group. Let $R$ be a $G$-graded ring and $n$ a natural number. Then, $R$ is graded U-nil clean if and only if $T_n(R)(\sigma)$ is graded U-nil clean, for any $\sigma \in G^n$.
\end{thm}

\begin{proof}
Let $n$ be a natural number and suppose that $T_n(R)(\sigma)$ is graded U-nil clean for any $\sigma \in G^n$. Let's choose $\sigma=(e,e,\dots,e)$ and consider the epimorphism of rings $\varphi: T_n(R)(\sigma) \longrightarrow R$ defined by $\varphi((a_{ij})_{n\times n})=a_{11}$. It is clear that the image of $T_n(R)_g(\sigma)$ by $\varphi$ is $R_g$ for any $g\in G$; hence $\varphi$ is a degree-preserving homomorphism. According to Lemma \ref{l1}, $\ker(\varphi)$ is a graded ideal of $T_n(R)(\sigma)$.  It follows, the rings $R$ and $T_n(R)(\sigma)/\ker(\varphi)$ are graded-isomorphic. Furthermore, according to Theorem \ref{t1}, $T_n(R)(\sigma)/\ker(\varphi)$ is graded U-nil clean, and thus $R$ is graded U-nil clean as well.

Conversely, suppose that $R$ is graded U-nil clean and let $\sigma \in G^n$. Let's consider the ideal $I$ of the ring $T=T_n(R)(\sigma)$ formed by matrices of $T$ with zeros in the main diagonal. Let's prove that $I$ is a graded ideal of $T$. Let $\sum_{g\in G} A_g \in I$, with $A_g\in T_g$ ($T_g$ is the $g$-th component of the $G$-graded ring $T$) for every $g\in G$; our  aim is to show that $A_g \in I$ for any $g\in G$. Indeed, since the sum $\sum_{g\in G} R_g$ is direct, then every element of the main diagonal of $A_g$ is nil for each $g\in G$. Consequently, $A_g\in I$ for each $g\in G$;  thus $I$ is a graded ideal. On the other hand, $I$ is a nilpotent ideal and in particular a graded-nil ideal of $T$. Now, let's consider the ring epimorphism $\epsilon: T \longrightarrow \prod_{1=1}^n R$ given by $\epsilon((a_{ij})_{n\times n })=(a_{11},a_{22},\dots,a_{nn})$. Let $A \in T_g$ where $g\in G$. Since $G$ is an abelian group, then elements of the main diagonal of $A_g$ are all from $R_g$. Hence, $\epsilon$ is a graded-homomorphism. Further, the kernel of $\epsilon$ is exactly the ideal $I$; therefore, the two rings $T/I$ and $\prod_{1=1}^n R$ are graded-isomorphic. According to Proposition, \ref{p4} $\prod_{1=1}^n R$ is graded U-nil clean; it follows that $T/I$ is a graded U-nil clean ring. Now, since $I$ a graded-nil ideal, then by Theorem \ref{t1} we conclude that $T$ is graded U-nil clean.
\end{proof}

\begin{lemma}\label{l3}
Let $R=\oplus_{g\in G} R_g$ be G-graded  ring and $t \in Idem(R_e)\cap Z(R)$. If the two rings $tR$ and $(1-t)R$ are $G$-graded U-nil clean, then $R$ is graded U-nil clean.
\end{lemma}

\begin{proof}
Assume that $tR$ and $(1-t)R$ are graded U-nil clean. Since $t\in Z(R)$ then $1-t$ is a central element of $R$ as well; therefore we can apply peirce decomposition as follow: $R=tR\oplus(1-t)R$. On the other hand, the ring homomorphism $\epsilon:tR\times (1-t)R\longrightarrow R$ defined by $\epsilon((tx,(1-t)y)=tx+(1-t)y$ is a graded-isomorphism. According to Proposition \ref{p4}, the ring $tR\times (1-t)R$ is graded U-nil clean. Therefore, $R$ is graded U-nil clean.

\end{proof}

We recall that idempotent elements $f_1,f_2,\dots,f_n$ of a ring are called orthogonal if $f_if_j=0$ whenever $i\neq j$. Next, we give a generalization of Lemma \ref{l3} using again the peirce decomposition.

\begin{lemma}\label{l4}
Let $R=\oplus_{g\in G}R_g$ be a $G$-graded ring, and suppose that $1=f_1+\dots+f_n$ in $R$ where $f_i\in Idem(R_e)\cap Z(R)$ are orthogonal for $1\leq i \leq n$. If $f_iR$ is graded U-nil clean for each $i\in \{1,\dots,n\}$, then $R$ is graded U-nil clean.
\end{lemma}

\begin{thm}[Main result]
Let $R=\oplus_{g\in G}R_g$ be a $G$-graded ring. If $R$ is graded U-nil clean, then $M_n(R)(\sigma)$ is graded U-nil clean for every natural number $n$ and every $\sigma \in G^n$.
\end{thm}

\begin{proof}
Let $n$ be an arbitrary natural  number and $\sigma \in G^n$. It is clear that $e_{11},\dots,e_{nn} $ are orthogonal idempotents of $M_n(R)_e(\sigma)$, where $e_{ii}$ is the matrix having 1 on the $(i,i)$-position and 0 elsewhere. Moreover, $e_{ii}$ is in the center of $M_n(R)(\sigma)$ for each $1\leq i \leq n$, and we have that $e_{11}+\dots+e_{nn}=1$. On the other hand, the ring homomorphism $\phi_i:e_{ii}M_n(R)(\sigma) \longrightarrow R$ given by $\phi(e_{ii}(a_{kj})_{n\times n})=a_{ii}$ is a graded-isomorphism for each $i\in \{1,\dots,n\}$. Hence, $e_{ii}M_n(R)(\sigma)$ is a graded U-nil clean ring for every $i\in \{1,\dots,n\}$. Therefore, according to Lemma \ref{l4} $M_n(R)(\sigma)$ is $G$-graded U-nil clean.
\end{proof}

\end{document}